\documentclass[12pt,reqno]{amsart} 
\usepackage{verbatim}
\usepackage{amssymb}

\usepackage{enumerate}
\usepackage[active]{srcltx}
\numberwithin{equation}{section}

\usepackage{t1enc}
\usepackage[utf8x]{inputenc}

\newtheorem{theorem}{Theorem}[section]
\newtheorem{lemma}[theorem]{Lemma}
\newtheorem{problem}[theorem]{Problem}
\newtheorem{example}[theorem]{Example}

\newtheorem{claim}{Claim}[theorem]

\theoremstyle{definition}
\newtheorem{definition}[theorem]{Definition}

 \theoremstyle{remark}

\newcommand{\mc}[1]{\mathcal{#1}}
\newcommand{\mbb}[1]{\mathbb{#1}}
\newcommand{\mb}[1]{\mathbf{#1}}
\newcommand{\mf}[1]{\mathfrak{#1}}

\newcommand{\setm}{\setminus}
\newcommand{\empt}{\emptyset}
\newcommand{\subs}{\subset}
\newcommand{\dom}{\operatorname{dom}}
\newcommand{\ran}{\operatorname{ran}}
\def\<{\left\langle}
\def\>{\right\rangle}
\def\br#1;#2;{\bigl[ {#1} \bigr]^ {#2} }

\newcommand{\oo}{{\omega}_1}
\newcommand{\oot}{{\omega}_2}

\DeclareMathOperator{\cf}{cf}

\newcommand{\lordop}{\operatorname{ord}_{\ell}}

\newcommand{\lord}[1]{\lordop(#1)}

\newcommand{\rordop}{\operatorname{ord}_{r}}

\newcommand{\rord}[1]{\rordop(#1)}

\newcommand{\httop}{\operatorname{ht}}

\newcommand{\htt}[1]{\httop(#1)}

\newcommand{\baire}[1]{\operatorname{B}(#1)}

\newcommand{\rank}{\operatorname{rk}}

\newcommand{\Pko}{P^{\kappa}}
\newcommand{\Pkoo}{P^{\kappa}_0}
\newcommand{\pcalk}{{\mc P}^{\kappa}}

\newcommand{\qsd}{Q}

\newcommand{\anfg}{\<A,n,f,g\>}
\def\anfgi#1;{\<A^{#1},n^{#1},f^{#1},g^{#1}\>}
\newcommand{\bde}{\<B,d,e\>}
\def\bdei#1;{\<B^{#1},d^{#1},e^{#1}\>}

\def\ap{A^p}
\def\fp{f^p}

\def\np{n^p}

\def\up{U^p}
\def\sbar{\overline{\sigma}}
\def\sstr{{\sigma}^*}

\newcommand{\id}{\operatorname{id}}

\makeatletter
\newcommand{\mydotminus}{\mathbin{\mathpalette\prc@inner\relax}}
\newcommand{\prc@inner}[2]{%
  \vbox{\offinterlineskip\m@th
    \ialign{%
      ##\cr
      \hidewidth\raisebox{-1.5\height}[0pt][0pt]{$#1.$}\hidewidth\cr
      $#1-$\cr
    }%
  }%
}
\makeatother

\author[L. Soukup]{Lajos Soukup}
\thanks
  {The first author was 
supported by   National Research, Development and Innovation Office – NKFIH, K113047.
   } 
\address{Alfréd Rényi Institute of Mathematics, Hungarian Academy of Sciences}
\email{soukup@renyi.hu
}

   \author[A. Stanley] {Adrienne Stanley}
\address{Department of Mathematics, University of Northern Iowa}
\email{adrienne.stanley@uni.edu}
\thanks{The second author was 
supported by  Fulbright Scholar Program.}

\title[Left-Separating Order Types]
   {Left-Separating Order Types}
\date{\today}

\begin{document}

\begin{abstract}
A well ordering $\prec$ of a topological space $X$ is {\em left-separating}
if $\{x'\in X: x'\prec x\}$ is 
 closed in $X$ for any $x\in X$.
A space is {\em  left-separated} if it has a left-separating well-ordering.
The {\em left-separating type} $\lord{X}$ of a left-separated space $X$
is the minimum of the order types of the left-separating well-orderings of $X$.

We prove that 
 \begin{enumerate}[(1)]
 
 \item 
if  ${\kappa}$ is a regular cardinal, then 
for each ordinal ${\alpha}<{\kappa}^+$
  there is a $T_2$ space $X$ with  
 $\lord X={\kappa}\cdot {\alpha}$;
\item if ${\kappa}={\lambda}^+$  and $\cf({\lambda})={\lambda}>{\omega}$,
then 
for each ordinal ${\alpha}<{\kappa}^+$
  there is a 0-dimensional space $X$ with  
 $\lord X={\kappa}\cdot {\alpha}$;
\item
 if ${\kappa}=2^{\omega}$ or ${\kappa}=\beth_{{\beta}+1}$, where $\cf({\beta})={\omega}$,
 then 
for each ordinal ${\alpha}<{\kappa}^+$
  there is a locally compact, locally countable, 0-dimensional space $X$ with  
 $\lord X={\kappa}\cdot {\alpha}$.
\end{enumerate}

The union of two left-separated spaces is not necessarily left-separated. 
We show, however,  that if $X$ is a countably tight  space, 
$X=Y\cup Z$, $\lord{Y}<\oo\cdot \omega$
and $\lord Z<\oo\cdot \omega$, then $X$ is also left-separated and 
\begin{align*}
 \lord{X}\le \lord Y+ \lord Z.
\end{align*}

We prove that 
it is consistent that 
there is a first countable, 0-dimensional space $X$, which is not left-separated, but 
there is a c.c.c poset $Q$ such that 
\begin{align*}
 V^Q\models \lord{X}=\oo\cdot \omega.
\end{align*}
However, if 
 $X$ is a topological space and $Q$ is a c.c.c poset such that 
\begin{align*}
 V^Q\models \lord{X}<\oo\cdot \omega,
\end{align*}
then  $X$ is left-separated  even in $V$.
\end{abstract}

\maketitle

\section{Introduction}

The notion of left- and right-separated spaces was introduced by Hajnal and Juhász \cite{HJ0}.

A well ordering $\prec$ of a topological space $X$ is 
{\em left-separating (right-separating) }
if $\{x\in X: x\prec y\}$ is 
 closed (respectively, open)   in $X$ for any $y\in X$.
A space is {\em  left-separated  (right-separated)} if it has a left-separating
(respectively, right-separating)
well-ordering.
The {\em left separating type} $\lord{X}$ of a left-separated space $X$
is the minimum of the order types of the left-separating well orders of $X$.
The {\em right separating type} $\rord{X}$ of a right-separated space $X$
is defined analogously: it is the minimum of the order types of the right-separating well orders 
of $X$. We write $\lord X=\infty$   ($\rord X=\infty$)
if $X$ is not left-separated  (respectively, not right-separated).

A space $X$  is {\em scattered} if
every non-empty subspace of $X$ has an isolated point.
The Cantor-Bendixson height of a scattered space $X$
will be denoted by $\htt{X}$. 

A space is scattered if and only if it is right-separated.
Since $\htt{X}\le \rord{X}$ holds for scattered 
spaces, and  $\htt{{\omega}^{\alpha}}={\alpha}$ for any ordinal ${\alpha}>0$, 
where ${\omega}^{\alpha}$ denotes ordinal exponentiation,  
 for each infinite ordinal ${\alpha}$ we have $|{\omega}^{\alpha}|=|{\alpha}|$
and $\rord{{\omega}^{\alpha}}\ge {\alpha}$.
So the right-separating types of spaces of cardinality ${\kappa}$ are unbounded
in ${\kappa}^+$.

In section \ref{sc:lord}, 
we discuss the same problem for left-separated spaces:
\begin{enumerate}[\quad]
 \item[] {\em Let ${\kappa}$ be an infinite cardinal and ${\alpha}<{\kappa}^+$.
Is there a left-separated regular (or Hausdorff) space $X$ of size $\kappa$ 
with 
$\lord X>{\alpha}$  (or with $\lord X={\alpha}$)?} 
 \end{enumerate}
The answer is clearly {\bf no} for ${\kappa}={\omega}$: every countable $T_1$
space is left-separated in type ${\omega}$.
In theorems \ref{tm:lord1} and  \ref{tm:lord2} 
we give a partial answer for cases where  ${\kappa}>{\omega}$. 

\begin{theorem}\label{tm:lord1}
If   ${\kappa}$ is  a regular, uncountable cardinal, and  
${\alpha}<{\kappa}^+$ is  an ordinal, then 
 there is a first countable,  scattered $T_2$  space 
$Y$ such that $\lord Y={\kappa}\cdot{\alpha}$.
\end{theorem}

To find regular spaces with large left-separating types
we need extra assumptions.

\begin{theorem}\label{tm:lord2}
If ${\kappa}>{\omega}_1$ is a regular cardinal, and there is 
a non-reflection stationary set $S\subs {\kappa}$, 
then for each ${\alpha}<{\kappa}^+$ there is a 0-dimensional space $Y$
with $\lord{Y}={\kappa}\cdot {\alpha}$.
\end{theorem}

It is well-known that if an uncountable right separated space $X$  is locally countable, 
then $\rord{X}=|X|$. The next theorem shows that the analogue of this statement is not true
for left separating spaces.

\begin{theorem}\label{tm:beth}
If ${\kappa}=\beth_{{\alpha}\dotplus1}$, where ${\alpha}=0$ or $\cf({\alpha})={\omega}$,
then  for each ${\alpha}<{\kappa}^+$
there is a locally compact, locally countable 0-dimensional space 
$X$ with $\lord{X}={\kappa}\cdot {\alpha}$.
\end{theorem}

The following question remained open. 

\begin{problem}
Is it true, in ZFC, that for each uncountable cardinal ${\kappa}$
and for each ordinal ${\alpha}<{\kappa}$ there is a 0-dimensional, $T_2$ space
$X$ with ${\alpha}<\lord{X}<{\kappa}^+$?  In particular, what about ${\kappa}={\omega}_1$
and ${\kappa}=\aleph_{\omega}$? 
\end{problem}

The union of two scattered space is clearly scattered, so the class of right-separated spaces
is closed under finite union.  In section \ref{sc:add} we investigate the 
same question for left-separated spaces.  
First we show that {\em there are two dense left-separated subsets 
$A$ and $B$ of $2^{\mf c}$ such that 
$A\cup B$ is not left-separated} (see Example \ref{ex:union}). 
On the other hand, we also get some positive results.

\begin{theorem}\label{tm:union}
If $A$ and $B$ are left-separated spaces, 
$t(A\cup B) = \omega$ and 
 $\lord {A}+\lord{B}<\oo\cdot {\omega}$
, then 
 $A\cup B$ is left-separated and  $\lord{A\cup B} <\oo \cdot {\omega}$.
\end{theorem}

\medskip

Properties ``left-separated and ``right-separated'' are upward absolute
because a left-separating (respectively, right-separating) well-order
remains left-separating (respectively, right-separating) 
in any extension of the ground model.

Property ``right-separated'' is also downward absolute because
if a space is scattered in any extension that it is also scattered 
in the ground model. 
So the  property ``right-separated'' is absolute.

What about  property ``left-separated''?  Since a countable $T_1$ space is 
automatically left-separated, we  consider only cardinal preserving extensions.
Moreover, a subspace $S\subs {\omega}_1 $
 is left-separated if and only if it is stationary, so it is reasonable to  consider only 
 stationary preserving forcing extensions, in particular, c.c.c generic extensions.
 
 In section \ref{sc:pres}, we investigate if c.c.c generic extensions preserve
  ``not left-separated'' property.  We prove the following result. 

 \begin{theorem}\label{tm:presccc}
(1) If $X$ is a topological space, $Q$ is a c.c.c poset such that 
\begin{align*}
 V^Q\models \lord{X}<\oo\cdot \omega,
\end{align*}
then  $X$ is left-separated  even in $V$, moreover
\begin{align*}
 \lord{X}^V=\lord{X}^{V^Q}.      
\end{align*}
(2) It is consistent that there is a $0$-dimensional, first countable topological  space $X$
 such that  
\begin{align*}
\lord{X}=\infty, 
\end{align*}
but 
\begin{align*}
 V^Q\models \lord{X}=\oo\cdot \omega
\end{align*}
for some c.c.c poset $Q$.
 \end{theorem}

The following questions remained open. 

\begin{problem}
Is there, in ZFC, a   $0$-dimensional, first countable topological  space $X$
which is not left-separated, but which becomes left-separated in some
c.c.c generic extension?  
\end{problem}

\begin{problem}
Is it consistent that there is a left-separated space $X$ 
such that 
\begin{align*}
 \lord{X}^{V^Q}<\lord{X}^V
\end{align*}
for some c.c.c.(proper) poset $Q$?
\end{problem}

\noindent{\bf Notation.}

If $P$ is a property, we write $X\subs ^P Y$ to mean that ``$X\subs Y$ and $X$
 has property $P$''.  In particular, 
 \begin{itemize}
  \item  if ${\kappa}$ is a cardinal and  
 $S, T\subs {\kappa}$, then we write $S\subs^{stat}T$ if $S\subs T$
 and $S$ is stationary in ${\kappa}$;
\item if $X$ is a topological space and $A$ is a set, then we write
$A\subs^{closed}X$ if $A\subs X$ and $A$ is closed in $X$.
  \end{itemize}

If ${\alpha}$ and ${\beta}$ are ordinal, then
${\alpha}\dotplus {\beta}$ and ${\alpha}\cdot {\beta}$ denote ordinal addition
  and ordinal multiplication, respectively.

 Given an ordinal ${\alpha}>0$ let 
\begin{align*}
{\alpha}\mydotminus1=\left\{
\begin{array}{ll}
 {\alpha}&\text{if ${\alpha}$ is a limit ordinal,}\\
 {\beta}&\text{if ${\alpha}={\beta}\oplus 1$ is a successor ordinal.}
\end{array}
\right.
\end{align*}

\section{The values of the $\lordop$ function}
\label{sc:lord}

To prove Theorems \ref{tm:lord1}, \ref{tm:lord2} and   \ref{tm:beth} 
we need Lemma \ref{lm:appr-exact} and  Theorem \ref{tm:ordl-bounds} below. 

\begin{lemma}\label{lm:appr-exact}
If ${\kappa}$ is an infinite cardinal and 
$${\kappa}\cdot{\alpha}\le \lord{X}\le {\kappa}\cdot({\alpha}\dotplus 1),$$  then there is 
a closed subspace $Y\subs^{closed}X$ with $\lord{Y}={\kappa}\cdot{\alpha}$.
\end{lemma}

\begin{proof}
Fix a bijection $f:X\to \lord{X} $
such that $f^{-1}{\zeta}\subs^{closed}X$ for each ${\zeta}<\lord{X}$.

Consider the subspace $Y=f^{-1}{\kappa}\cdot{\alpha}$.

The function  $f$ witnesses that $\lord{Y}\le {\kappa}\cdot{\alpha}$.
If $\lord{Y}= {\kappa}\cdot{\alpha}$, then we are done.

If $\lord{Y}={\beta}<{\kappa}\cdot{\alpha}$, then $\lord{X}\le \lord{Y}\dotplus \lord{X\setm Y}\le
{\beta}\dotplus {\kappa}\le {\kappa}\cdot{\alpha}$.  Thus $\lord{X}={\kappa}\cdot{\alpha}$ and so $Y=X$
satisfies the requirements.
\end{proof}

\begin{definition}
 If $\mc T=\<T,\preceq\>$  is a well-founded partially ordered set, then we 
 define the {\em rank function} 
\begin{equation*}
 \rank_{\mc T}:T\to  \mb {On}
\end{equation*}
by the recursive formula  
$$\rank_{\mc T}(t)=\sup\{\rank_{\mc T}(s)\dotplus 1:s\prec t\in T\}.$$
Hence the minimal elements of $\mc T$ have rank $0$.
The rank of $\mc T$ is defined as follows:
\begin{align*}
\rank({\mc T})=&\sup\{\rank_{\mc T}(t)\dotplus 1: t\in T\}=\\
&\min\big(\mb{On}\setm \ran(\rank_{\mc T})\big)=\ran(\rank_{\mc T}). 
\end{align*}
\end{definition}

The next theorem makes us possible the find upper and lower bounds of $\lord{X}$
for  certain spaces.

\begin{theorem}\label{tm:ordl-bounds}
Assume that 
\begin{enumerate}[(S1)]
 \item ${\kappa}\ge \cf({\kappa})>{\omega}$ is a   cardinal,
 \item $\mc T=\<T,\preceq\>$ is a well-founded partially ordered set with $|T|\le {\kappa}$,
 \item  $\mc X=\<X,\tau\>$ is a topological space with  $|X|={\kappa}$,
 \item $\{X_t:t\in T\}\subs \br X;{\kappa};$ is a partition of $X$, 
 \end{enumerate}

\smallskip
\noindent (a)
If  
\begin{enumerate}[(S1)]
\addtocounter{enumi}{4}
 \item for each $s\in T$ every   $x\in X_s$  has a neighborhood $U(x)$ such that 
$$U(x)\setm \{x\}\subs \bigcup\{X_t:s\prec t\in T\},$$
\end{enumerate}
then $\mc X$ is left-separated and 
\begin{equation*}
\text{$ \lord{\mc X}\le {\kappa}\cdot \rank({\mc T})$.} 
\end{equation*}

\smallskip
\noindent (b)
If 
\begin{enumerate}[(S1)]
\addtocounter{enumi}{5}
\item \label{S6} for each $s\prec t\in T$ and $S\in \br X_t;{\kappa};$ 
there is $R\in \br  S;<\cf({\kappa});$ such that  
$|\overline{R}^{\tau}\cap X_s|={\kappa}$,
\end{enumerate}
then 
\begin{equation*}
\text{${\kappa}\cdot(\rank({\mc T})\dot -1)\le \lord{\mc X}$.} 
\end{equation*}

\smallskip
\noindent (c)
If 
\begin{enumerate}[(S1)]
\addtocounter{enumi}{6}
\item \label{S7}
\begin{enumerate}[(i)]
 \item 
$X\subs {\kappa}=\cf({\kappa})$, \item  the topology $\tau$ refines 
the order topology 
on ${\kappa}$, and \item 
 for each $s\prec t\in T$ and $S\in \br X_t;{\kappa};$ we have 
$\overline{S}^{\tau}\cap X_s\subs^{stat}{\kappa}$,
\end{enumerate}

\end{enumerate}
then 
\begin{equation*}
\text{${\kappa}\cdot(\rank({\mc T})\mydotminus1)\le \lord{\mc X}$.} 
\end{equation*}
\end{theorem}
To prove (b) and (c) we will need the following easy lemma. 

\begin{lemma}\label{cl:rhorange}
If ${\alpha}\in \mb{On}$ and  ${\sigma}:{\alpha}\to \mb{On}$
is a function such that 
\begin{displaymath}
 {\sigma}({\zeta})\dotplus{\kappa}\le {\sigma}({\xi}) \text{ for each ${\zeta}<{\xi}<{\alpha}$,}
\end{displaymath}
then $${\kappa}\cdot ({\alpha}\mydotminus1)\le \sup {\sigma}''{\alpha}.$$
\end{lemma}
\begin{proof}[Proof of the Claim]
By transfinite induction it is a straightforward to show that 
${\kappa}\cdot ({\xi}\mydotminus 1)\le  {\sigma}({\xi})$ for ${\xi}<{\alpha}$.
\end{proof}

\begin{proof}[Proof of Theorem  \ref{tm:ordl-bounds}](a)
 Write ${\alpha}=\rank(\mc T)$.

Consider a bijection  $f:X\to {\kappa}\cdot {\alpha}$
such that for all ${\zeta}<{\alpha}$ 
\begin{align*}
\{f(x): x\in \bigcup_{\rank_{{\mc T}}(t)={\zeta}} X_t \}= 
[{\kappa}\cdot {\zeta},{\kappa}\cdot {\zeta}\dotplus {\kappa}).
\end{align*}
Then $f^{-1}{\delta}\subs^{closed}X$ for all ${\delta}<{\kappa}\cdot {\alpha}$
 by (S5) because $s\prec t$ implies 
$\rank_{\mc T}(s)<\rank_{\mc T}(t)$
and so $f(x)={\delta}$ implies $U(x)\cap f^{-1}{\delta}=\empt$.
Thus  $\lord{\mc X}\le {\kappa}\cdot {\alpha}$.

\bigskip
\noindent{(b) and (c).}

 Assume that $\lord{\mc X}={\beta}$, and fix a bijection 
$ f:X\to {\beta}$
such that for each ${\beta}'<{\beta}$
$$f^{-1}{\beta}'\subs^{\rm closed} X.$$

Define the function $\rho:\rank(\mc T)\to \bf {On}$ as follows:
\begin{align*}
 \rho({\zeta})=\min\{{\beta}'\le {\beta}:\exists t\in T( \rank_{\mc T}(t)={\zeta}\land (|f^{-1}{\beta}'\cap X_t)|= {\kappa})\}.
\end{align*}
We have  $\cf(\rho({\zeta}))=\cf({\kappa})$ for all ${\zeta}<\rank(\mc T)$.

\begin{claim}\label{cL:rhojumpS6}
If  (S\ref{S6})
holds and 
${\zeta}<{\xi}<rk(T)$, then $\rho({\zeta})<\rho({\xi})$.
\end{claim}
\begin{proof}[Proof of the Claim]
Fix  $t\in T$ such that $\rank(t)={\xi}$ and $$S= (f^{-1}{\rho({\xi})})\cap X_t$$
has cardinality ${\kappa}$.

Pick $s\prec t$ with $\rank(s)={\zeta}$.
By (S\ref{S6}), there is $T\in \br S;<\cf({\kappa});$
such that $\overline {S}\cap X_s$ has cardinality ${\kappa}$.

Since $\cf(\rho({\xi}))=cf({\kappa})$, there is $\eta<\rho({\xi})$
such that $S\subs f^{-1}\eta$.

Since $f^{-1}\eta\subs^{closed}X$, we have $\overline {S}\cap X_s\subs f^{-1}\eta$.

Thus $\rho({\zeta})\le \eta$.
\end{proof}

\begin{claim}\label{cL:rhojumpS7}
If (S7)  holds and 
${\zeta}<{\xi}<rk(T)$, then $\rho({\zeta})<\rho({\xi})$.
\end{claim}
\begin{proof}[Proof of the Claim]
Write ${\delta}=\rho({\xi})$ and fix $t\in T$
such that $\rank(t)={\xi}$  and  $$Y_t=f^{-1}{\delta}\cap X_t$$ has cardinality ${\kappa}$.

Pick $s\prec t$ with $rank(s)={\zeta}$.
By (S\ref{S7}),  $$Y_s=\overline {f^{-1}{\delta}\cap X_t}\cap X_s\subs^{stat}  X_s.$$

Since $f^{-1}{\delta}\subs ^{closed}X$, we have   
\begin{equation*}
Y_s\subs \overline{(f^{-1}{\delta}\cap X_t)}^\tau\subs f^{-1}{\delta}.  
\end{equation*}
Thus $\rho({\zeta})\le {\delta}$.

Assume on the contrary that $\rho({\zeta})=\rho({\xi})={\delta}$.

Fix a club subset  $C =\{{\gamma}_{\nu}:{\nu}<{\kappa}\}$  of $\delta$,
and consider the set 
\begin{align*}
 D=\{{\nu}<{\kappa}: f[Y_s\cap {\nu}]= f[Y_s]\cap {\gamma}_{\nu} \land
 f[Y_t\cap {\nu}]= f[Y_t]\cap {\gamma}_{\nu}\}.
\end{align*}
$D$ is club because the sets $f^{-1}\eta\cap Y_s$ and 
 $f^{-1}\eta\cap Y_t$ have cardinalities $<{\kappa}$ for each $\eta<{\delta}$.
Since $Y_s$ is stationary, there is ${\nu}\in D\cap Y_s$.

Then ${\nu}\in \overline{Y_t\cap {\nu}}$ because 
the topology $\tau$ refines the order topology on ${\kappa}$.

But $Y_t\cap {\nu}\subs f^{-1}{\gamma}_{\nu}\subs ^{closed} X$
because $f[Y_t\cap {\nu}]= f[Y_t]\cap {\gamma}_{\nu}$,
and ${\nu}\notin f^{-1}{\gamma}_{\nu}$ 
because $f[Y_s\cap {\nu}]= f[Y_s]\cap {\gamma}_{\nu}$.
Thus $f^{-1}{\gamma}_{\nu}$ is not closed in $X$.
Contradiction,  $\rho({\xi})=\rho({\zeta})$ is not possible.
So the claim holds.
\end{proof}
By Claims \ref{cL:rhojumpS6} and \ref{cL:rhojumpS7}
we have $\rho({\zeta})<\rho({\xi})$ for ${\zeta}<{\xi}<{\alpha}=\rank(\mc T)$.
Since $\cf(\rho({\xi})={\kappa}$, it follows that 
$\rho({\zeta})\dotplus {\kappa}\le \rho({\xi})$.

Thus we can apply Lemma \ref{cl:rhorange} to derive that 
$$
{\kappa}\cdot (\rank(\mc T)\mydotminus 1)=
{\kappa}\cdot ({\alpha}\mydotminus 1)\le \sup \rho''{\alpha}\le {\beta}=\lord{\mc X},$$
which was to be proved. 
\end{proof}

%
%

%
\begin{proof}[Proof of Theorem \ref{tm:lord1}]
Let $\{X_{\zeta}:{\zeta}<{\alpha}\dotplus1\}$ be a family of pairwise disjoint 
stationary subsets of $E^{\omega}_{\kappa}=\{{\nu}<{\kappa}: \cf({\nu})={\omega}\}$. 
%

Write
\begin{align*}
X_{>{\xi}}=\bigcup \{X_{\zeta}:{\zeta}>{\xi}\}.
\end{align*}

We define 
the topology  $\tau$ on $X$ as follows.

For ${\xi}<{\alpha}\dotplus1$ and ${\gamma}\in X_{\xi}$, for all $\beta<{\gamma}$
let 
\begin{equation*}
B_{{\gamma}}(\beta)=\{{\gamma}\}\cup \big((\beta,{\gamma})\cap X_{>{\xi}}\big). 
\end{equation*}
The neighborhood base of ${\gamma}$ in $X$ will be
\begin{equation*}
 \{B_{\gamma}(\beta):\beta<{\gamma}\}.
\end{equation*}

The topology $\tau$  is finer than the usual order topology 
of $\kappa$, so $\<X,\tau\>$ is a scattered $T_2$ space.
Since $cf({\nu})=\omega$ for all ${\nu}\in X$, the space is first countable.

We are to show that ${\kappa}$, $T={\alpha}\dotplus1$, $\{X_{\zeta}:{\zeta}<{\alpha}\dotplus1\}$ and 
$\<X,\tau\>$
satisfy (S1)--(S5) and (S7) from Theorem \ref{tm:ordl-bounds}.

(S1)--(S5) are clear from the construction.

To check (S7), assume that ${\zeta}<{\xi}<{\alpha}\dotplus1$ and $S\in \br X_{\xi};{\kappa};$.
Then 
\begin{equation*}
 \overline{S}\cap X_{{\zeta}}=\overline{S}^\tau\cap X_{\zeta}
\end{equation*}
by the definition of $\tau$, where $\overline{S}$
denotes the closure of $S$ in the order topology. So $\overline{S}^\tau\cap X_{\zeta}$
is stationary because 
$X_{\zeta}\subs^{stat} {\kappa} $ and 
$\overline{S}\subs^{\rm club}{\kappa}$. 

So we can apply Theorem \ref{tm:ordl-bounds}(a) and (c)  to obtain 
\begin{align*}
 {\kappa}\cdot {\alpha}={\kappa}\cdot ( ({\alpha}\dotplus 1))\mydotminus1)\le
 \lord{\mc X}\le  {\kappa}\cdot ( ({\alpha}\dotplus 1))
\end{align*}
Then, applying  Lemma \ref {lm:appr-exact}
we obtain a closed subspace $Y\subs^{closed}X$
such that $\lord{Y}={\kappa}\cdot {\alpha}$.
\end{proof}

\newcommand{\hfd}[1]{$#1$-HFD${}_w$}

\begin{definition}
A subspace $W=\{w_{\nu}:{\nu}<{\kappa}\}\subs 2^{\kappa}$ is a
{\em \hfd{{\kappa}}} iff 
\begin{align*}
 \forall A\in \br W;{\kappa};\ \exists B\in\br  A;<{\kappa};\ &\exists {\gamma}<{\kappa}\\
 &\{w\restriction {\kappa}\setm {\gamma}:w\in B\}\subs^{dense}2^{{\kappa}\setm {\gamma}}.
\end{align*}
\end{definition}

\begin{theorem}\label{tm:leftfromheredsep}
Assume that ${\kappa}>{\omega}$ is a regular cardinal, and there is an \hfd{{\kappa}}
space $W=\{w_{\nu}:{\nu}<{\kappa}\}\subs 2^{\kappa}$.

Then for each  ordinal ${\alpha}<{\kappa}^+$
there is a  space $Y\subs 2^{\kappa}$ such that 
$\lord{Y}= {\kappa}\cdot {\alpha}$.
\end{theorem}

\begin{proof}[Proof of Theorem \ref{tm:lord2} from Theorem \ref{tm:leftfromheredsep}]
 Rinot \cite[Main Result]{rinot}  proved that 
if ${\kappa}>{\omega}_1$ is  a regular cardinal, and there is 
a non-reflecting stationary subset $S\subs {\kappa}$, then 
$Pr_1({\kappa},{\kappa},{\kappa},{\omega})$ holds
(see \cite{rinot} for the definition of $Pr_1({\kappa},{\kappa},{\kappa},{\omega})$).

Fix a   function 
$f:{\kappa}\times {\kappa}\to {\kappa}$ witnessing   
$Pr_1({\kappa},{\kappa},{\kappa},{\omega})$, and 
define $W=\{w_{\nu}:{\nu}<{\kappa}\}\subs 2^{\kappa}$
as follows:
\begin{align*}
w_{\nu}({\mu})=1 \text{ iff } f({\nu},{\mu})=1. 
\end{align*}
Then $W$ is a \hfd{{\kappa}} space.
So we can apply Theorem \ref{tm:leftfromheredsep}.
  \end{proof}

\begin{proof}[Proof of Theorem \ref{tm:leftfromheredsep}]
Fix an enumeration $\{w_{\nu}:{\nu}<{\kappa}\}$ of $W$. 

Let $\{K_s:s<{\alpha}\dotplus 1\}\subs \br {\kappa};{\kappa};$ be a partition of 
${\kappa}$.

Define $X=\{x_{\nu}:{\nu}<{\kappa}\}\subs 2^{\kappa}$ as follows:
if ${\nu}\in K_s$ and ${\mu}\in K_t$, then
\begin{align*}
 x_{\nu}({\mu})=\left \{
 \begin{array}{ll}\tag{$\dag$}
 1&\text{if ${\nu}={\mu},$}\\\\
 0&\text{if ${\nu}>{\mu},$}\\\\
 0&\text{if ${\nu}\ne {\mu}$ and $s\le t$}\\\\
 w_{\nu}({\mu})&\text{if ${\nu}<{\mu}$ and   $s>t$.}\\
 \end{array}
\right .
\end{align*}
Let $X_s=\{x_{\nu}:{\nu}\in K_s\}$ for $s<{\alpha}\dotplus1$.

We are to show that ${\kappa}$, $T={\alpha}\dotplus1$, $\{X_{\zeta}:{\zeta}<{\alpha}\dotplus1\}$ 
and  $X$
satisfy (S1)--(S6) from Theorem \ref{tm:ordl-bounds}.

(S1)--(S4) are clear. For $x_{\mu}\in X$
the open set $$U(x_{\mu})=\{x_{\nu}: x_{\nu}({\mu})=1\}$$
witnesses property (S5) by ($\dag$).

Next we check (S6).

\begin{claim}
If  $s< t\in T={\alpha}\dotplus 1$ and $S\in \br X_t;{\kappa};$,
then there is $R\in \br S;<{\kappa};$ such that 
$|\overline{R}\cap X_s|={\kappa}$.
\end{claim}

\begin{proof}[Proof of the Claim]
Assume on the contrary that for each ${\xi}<{\kappa}$ we have 
$$|\overline{\{x_{\nu}\in S :  {\nu}<{\xi} \}}\cap X_s|<{\kappa},$$
and so 
we can pick $\varepsilon_{\xi}\in Fn({\kappa},2)$
and  ${\mu}_{\xi}\ge{\xi}$ and 
\begin{align*}\tag{$\ddag$}
 x_{{\mu}_{\xi}}\in [\varepsilon_{\xi}]\cap X_s
 \ \land\ [\varepsilon_{\xi}]\cap  \{x_{\nu}\in S :  {\nu}<{\xi} \}=\empt
\end{align*}
By thinning out the sequences we can assume that $\{\varepsilon_{\xi}:{\xi}<{\kappa}\}$
form a $\Delta$-system with kernel $\varepsilon$.

Pick ${\zeta}$ such that 
${\mu}_{\zeta}>\max \dom(\varepsilon)$.
Then $x_{{\mu}_{\zeta}}({\nu})=0$ for ${\nu}\in \dom (\varepsilon)$ by ($\dag$).
Thus $x_{{\mu}_{\zeta}}\in [\varepsilon_{\zeta}]$ implies that 
\begin{align*}\tag{$*$}
\text{$\varepsilon({\nu})=0$
for each ${\nu}\in \dom (\varepsilon)$.} 
\end{align*}

Let $J=\{{\nu}: x_{\nu}\in S\land {\nu}>\max \dom(\varepsilon)\}$.
Since $W$ is a \hfd{{\kappa}}, one can find $I\in  \br J;<{\kappa}; $
and ${\gamma}<{\kappa}$ such that 
\begin{align*}\tag{$\star$}
\{w_{\nu}\restriction {\kappa}\setm {\gamma}:{\nu}\in I\}
\subs^{dense}2^{{\kappa}\setm {\gamma}}
 \end{align*}

Pick ${\xi}<{\kappa}$ such that 
\begin{enumerate}[(1)]
\item $\{w_{\nu}:{\nu}\in I\}\subs \{w_{\nu}\in S: {\nu}<{\xi}\}$, and 
 \item $\dom(\varepsilon_{\xi}\setm \varepsilon)\subs {\kappa}\setm {\gamma}$, and
 \item $\sup I<\min(\dom(\varepsilon_{\xi}\setm \varepsilon))$.
\end{enumerate}
By ($\star$) and (2) there is ${\nu}\in I$ such that 
\begin{align*}\tag{$\circ$}
\varepsilon_{\xi}\setm \varepsilon\subs w_{\nu}.
\end{align*}

Since ${\nu}>\max \dom(\varepsilon)$, $x_{\nu}({\mu})=0$
for all ${\mu}\in \dom(\varepsilon)$ by ($\dag$). So, by ($*$), 
\begin{align*}\tag{$\bullet$}
\varepsilon\subs x_{\nu}. 
\end{align*}

Assume that ${\mu}\in \dom(\varepsilon_{\xi}\setm \varepsilon)$.
Then ${\mu}\in X_r$ for some $r<{\alpha}\dotplus1$. 

If $r<t$,  then  ${\nu}<{\mu}$ by (3), and so 
\begin{align*}\tag{$\vartriangle$}
x_{\nu}({\mu})=w_{\nu}({\mu})=\varepsilon_{\xi}({\mu}). 
\end{align*}
by ($\dag$) and $(\circ)$.

If $t\le r$, then $s<r$.  So $x_{\nu}({\mu})=0$ 
by ($\dag$),
and 
$\varepsilon_{\xi}({\mu})=x_{{\mu}_{\xi}}({\mu})=0$
by ($\ddag$) and $(\dag)$.
So
\begin{align*}\tag{$\blacktriangle$}
x_{\nu}({\mu})=\varepsilon_{\xi}({\mu}). 
\end{align*}
$(\circ)$, ($\vartriangle$) and ($\blacktriangle$) give 
$x_{\nu}\in [\varepsilon_{\xi}]$ which contradicts $(\ddag)$ and (1).
We proved the claim.
\end{proof}

So we can apply Theorem \ref{tm:ordl-bounds}(a) and (b)  to obtain 
\begin{align*}
 {\kappa}\cdot {\alpha}={\kappa}\cdot ( ({\alpha}\dotplus 1))\mydotminus1)\le
 \lord{\mc X}\le  {\kappa}\cdot ( ({\alpha}\dotplus 1))
\end{align*}
Then, applying  Lemma \ref {lm:appr-exact}
we obtain a closed subspace $Y\subs^{closed}X$
such that $\lord{Y}={\kappa}\cdot {\alpha}$.
\end{proof}

To prove Theorem  \ref{tm:beth} we need to recall some notions and notations from \cite{En}.
The discrete topological space on a set $A$ will be denoted by $D(A)$. 
The {\em Baire space of weight ${\lambda}$, $\baire{{\lambda}}$}, is 
the metric space $D({\lambda})^{\omega}$, see \cite[Example 4.2.12]{En}.
The metric topology of $\baire{{\lambda}}$ will be denoted by $\rho_{\lambda}$.

The following claim is well-known:
\begin{claim}\label{cl:baire-closure}
If ${\lambda}$ is a strong limit cardinal with $\cf({\lambda})={\omega}$
and  $C\subs  B({\lambda})$ has cardinality ${\lambda}^+$, then there is 
$D\in \br C;{\lambda};$ such that $|\overline{D}^{\rho_{\lambda}}|={\lambda}^{\omega}$.
\end{claim}

We say that $T\subs ^{<{\omega}}{{\kappa}}$ is a {\em well-founded tree} if  $T$ is closed
under initial segments,  and the poset $\mc T=\<T,\subs \>$ is a tree without 
infinite branches.

Given a cardinal ${\kappa}$
and an ordinal ${\alpha}<{\kappa}^+$
one can find a well-founded tree  $T_{\alpha}\subs ^{<{\omega}}{{\kappa}}$
such that $\rank(\<T_{\alpha},\supset\>)={\alpha}\dotplus1$.

\begin{proof}[Proof of Theorem \ref{tm:beth}]
 Let $T\subs ({\kappa})^{<{\omega}}$ be a well-founded tree such that 
 $rank(\mc T)={\alpha}+1$, where $\mc T=\<T,\supset\>$.

If $t\in T$ is not the root of the tree then denote $pred(t)$ the predecessor of
$t$ in $\mc T$.

Write ${\lambda}={\omega}$ if ${\kappa}=2^{\omega}$, and ${\lambda}=\beth_{\beta}$
if ${\kappa}=\beth_{{\beta}\dotplus1}$. Then ${\lambda}$ is  a strong limit cardinal
with $\cf({\lambda})={\omega}$.

Consider the family  
 $$\mc D=\{D\in \br {\baire{{\lambda}}};{\lambda};:
|\overline{D}^{\rho_\lambda}|={\kappa}\},$$ 
and fix an enumeration $\{D_{\nu}:{\nu}<2^{\omega}\}$
of $\mc D$.  
Let $\{\<t_i,{\nu}_i,{\zeta}_i\>:i<2^{\omega}\}$ be an enumeration of 
$T\times {\kappa}\times {\kappa}$.

We define, by transfinite recursion on ${\gamma}$,
a sequence $\<\<X^{\gamma},\tau_{\gamma}\>:{\gamma}\le {\kappa}\>$ 
 of topological spaces, and 
  partition $\<X^{\gamma}_t:t\in T\>$ of the set $X^{\gamma}$,
such that 
\begin{enumerate}[(a)]
\item $X^{\gamma}\subs B({\lambda})$ with $|X^{\gamma}|\le |{\gamma}|$,
\item the topology $\tau_{\gamma}$ refines the metric topology $\rho_{\lambda}$,
 \item $\<X^{\gamma},\tau_{\gamma}\>$ is locally compact and locally countable,
 \item $\tau_\beta=\tau_\gamma\cap \mc P(X^{\beta})$ for ${\beta}<{\gamma}$,
 \item every  $x\in X^{\gamma}_t$ has a neighborhood $U(x)\in \tau_{\gamma}$
such that 
\begin{equation*}
U(x)\setm \{x\}\subs \bigcup \{X^{\gamma}_s: s\subsetneq t\}. 
\end{equation*}
\end{enumerate}

\medskip

\noindent {\bf Case 1: ${\gamma}=0$.}  

Let $X^0=\empt$ and  $X^0_t=\empt$ for $t\in T$.

\medskip

\noindent {\bf Case 2}: {\em ${\gamma}$ is a limit ordinal.}  

Let $X^{\gamma}_t=\bigcup_{{\beta}<{\gamma}}X^{\beta}_t$ and 
$X^{\gamma}=\bigcup_{{\beta}<{\gamma}}X_{\beta}$, and the topology 
$\tau_{\gamma}$ is generated by the family 
$\bigcup_{{\beta}<{\gamma}}\tau_{\beta}$.

\medskip

\noindent {\bf Case 3: ${\gamma}={\beta}\dotplus1$.}  

Consider the triple $\<t_{\beta},{\nu}_{\beta}, {\zeta}_{\beta} \>$.

If $t_{\beta}$ is the root of the  tree:   just pick a new element
$x_{\beta}\in \baire{{\lambda}}\setm X^{\beta}$, let    
$X^{{\gamma}}_{t_{\beta}}=X^{{\beta}}_{t_\beta}\cup \{x_{\beta}\}$, 
let $X^{\gamma}_t=X^{\beta}_t$
for $t\in T\setm\{t_{\beta}\}$, and declare 
$x_{\beta}$ to be  isolated in $\tau_{\gamma}$.

Assume now that $t_{\beta}$ is not the root of the tree. Let  $s=pred(t_{\beta})$.

If $D_{{\nu}_{\beta}}\not\subs  X^{\beta}_s$, then do nothing:
let $X^{\gamma}=X^{\beta}$ and $X^{\gamma}_t=X^{\beta}_t$ for all $t\in T$.

Assume finally that $D_{{\nu}_{\beta}}\subs X^{\beta}_s$.
Then pick  
$x_{\beta}\in \overline{D_{{\nu}_{\beta}}}^{\rho_\lambda}\setm X^{\beta}$, and let 
$X^{\gamma}=X^{\beta}\cup \{x_{\beta}\}$, 
$X^{{\gamma}}_{t_{\beta}}=X^{{\beta}}_{t_\beta}\cup \{x_{\beta}\}$, and $X^{\gamma}_t=X^{\beta}_t$
for $t\in T\setm\{t_{\beta}\}$.

To define the topology $\tau_{\gamma}\supset \tau_{\beta}$
 pick a sequence 
$\{x_{{\beta},n}\}_{n\in {\omega}}\subs D_{{\nu}_{\beta}}$ converging to $x_{\beta}$
in the metric topology $\rho_{\lambda}$,
and pick pairwise disjoint compact open neighborhoods  $U_n$ of $x_{{\beta},n}$ 
in $\tau_{\beta}$
such that 
\begin{enumerate}[(a)]
\item $U_n\setm\{x_{{\beta},n}\}\subs \bigcup\{X^{\beta}_r:r\subsetneq s\}$, and 
 \item the sequence of sets , $\{U_n:n\in {\omega}\}$,  converges to $x_{\beta}$  
 in the metric topology $\rho_{\lambda}$.
\end{enumerate}
Let 
\begin{align*}
 B_n(x_{\beta})=\{x_{\beta}\}\cup \bigcup\nolimits_{m\ge n}U_m.
\end{align*}
for $n\in {\omega}$.
The topology $\tau_{\gamma}$ is generated by the family
\begin{align*}
 \tau_{\beta}\cup\{B_n(x_{\beta}):n\in {\omega}\}.
\end{align*}
It is clear from the construction that the space 
$\<X^{\gamma},\tau_{\gamma}\>$ has properties (a)--(e).

\bigskip

Let $\mc X=\<X^{\kappa},\tau_{\kappa}\>$

We are to show that ${\kappa}$, $\mc T$, $\{X_t:t\in T\}$ and 
$\mc X$
satisfy (S1)--(S6) from Theorem \ref{tm:ordl-bounds}.

(S1)--(S4) are clear. 
Property (e) yields (S5). 

Since there are no infinite chains in $\mc T$
it is enough to verify $(S6)$ for $t=pred (s)$.

So let $S\in \br X_t;{\kappa};$.
Then there is ${\nu}<{\kappa}$ such that $D_{\nu}\subs X_t$.

If $i$ is large enough and ${\nu}_i={\nu}$ and $D_{\nu}\subs X^i_t$,
then in the $i$th step we add a new point to $\overline {D_{\nu}}\cap X_s$.
So $|\overline{D_{\nu}}\cap X_s|={\kappa}$, which proves (S6).

\medskip

So we can apply Theorem \ref{tm:ordl-bounds}(a) and (b)  to obtain 
\begin{align*}
 {\kappa}\cdot {\alpha}={\kappa}\cdot ( ({\alpha}\dotplus 1))\mydotminus1)\le
 \lord{\mc X}\le  {\kappa}\cdot ( ({\alpha}\dotplus 1))
\end{align*}
Then, applying  Lemma \ref {lm:appr-exact}
we obtain a closed subspace $Y\subs^{closed}X$
such that $\lord{Y}={\kappa}\cdot {\alpha}$.
\end{proof}

\section{Additivity}\label{sc:add}

The union of two scattered spaces is clearly scattered. 
What about left-separated spaces?

We will present an example showing that the  union of two left-separated spaces
is not necessarily left-separated. We need the following lemma.

\begin{lemma}\label{lm:baire_not_left_sep}
If $X$ is a topological space, $X=A\cup B$, both $A$ and $B$ are dense in 
$X$, and   
\begin{enumerate}[($*$)]
 \item $|A|$-many  nowhere dense subsets of $B$ can not cover
a nonempty open subset of $B$,\end{enumerate}
 then $X$ is not left-separated. 
\end{lemma}

\begin{proof}
Assume on the contrary that $X$ is left-separated witnessed by the enumeration
$X=\{x_{\xi}:{\xi}<{\mu}\}$, i.e. $\{x_{\zeta}:{\zeta}<{\xi}\}\subs ^{closed}X$
for each ${\xi}<{\mu}$.

Let
\begin{align*}
{\nu}=\min\{{\mu}'\le {\mu}: U=int\overline{\{x_{\xi}:{\xi}<{\mu}'\}}\ne \empt\}. 
\end{align*}
Since $\{x_{\xi}:{\xi}<{\nu}\}$ is closed in $X$, it follows, that 
$A\cap U\subs \{x_{\xi}:{\xi}<{\nu}\}$.
Since $A\cap U$ is dense in $U$, by the minimality of ${\nu}$ we have 
$\cf {\nu}\le |A|$.
Let $\{{\nu}_n:n<\cf({\nu})\}$ be cofinal in ${\nu}$.

Then $\{a_{\xi}:{\xi}<{\nu}_n\}\cap B$ is nowhere dense for $n<\cf({\nu})$, so
by $(*)$ the set 
\begin{align*}
\{x_{\xi}:{\xi}<{\nu}\}\cap B=\bigcup_{n<\cf({\nu})}(\{x_{\xi}:{\xi}<{\nu}_n\}\cap B)
\end{align*}
can not cover a open subset of $B$. 

However the set $\{x_{\xi}:{\xi}<{\nu}\}$ is closed in $X$, which implies that  
$B\cap U\subs \{x_{\xi}:{\xi}<{\nu}\}$,  
 which is a contradiction. 
\end{proof}

The next example was constructed by Juhász, Soukup and Szentmiklóssy. 
\begin{example}\label{ex:union}
There are two dense left-separated subsets 
$A$ and $B$ of $2^{\mf c}$ such that 
$A\cup B$ is not left-separated. 
\end{example}

\begin{proof}
Let $A$ be a countable dense subsets of $2^{\mf c}$.
Let $B$ be a $G_{\delta}$-dense, left-separated subset of $2^{\mf c}$
with $|B|=\mf c$.  

Then $\lord{A}={\omega}$ and $\lord B=\mf c$.  
Since $B$ is $G_\delta$-dense, countably many nowhere dense 
sets can not cover a nonempty open subset of $B$.
So, by lemma \ref{lm:baire_not_left_sep}, the space $X=A\cup B$
is not left-separated. 
\end{proof}

\begin{theorem}
Let $(X_{i})_{i \in n}$ be a finite family of left-separated spaces such that 
for each $i \in n$, $\lord {X_{i}} = \oo$.  In addition, suppose $t(X) = \omega$ where $X = \bigcup_{i \in n} X_{i}$.
Then $X$ is left-separated and  $\lord{X} \leq \oo \cdot 
n$.
\end{theorem}

\begin{proof}
Let $(X_{i})_{i \in n}$ be as stated in the hypothesis of the theorem.  Let $X = 
\bigcup_{i \in n} X_{i}$. For each $i \in n$ and for each $A \in [n]^{i}$, 
define 
$$Z_{A} = {\left( \bigcap_{j \notin A} \overline{X_{j}} \right)} \setm {\left( 
\bigcup_{j \in A} \overline{X_{j}} \right) } .$$ For each $i \in n$, define 
$Y_{i} = \bigcup_{A \in [n]^{i}} Z_{A}$. 

\begin{claim} \label{1.1.2}
If $A,B\in [n]^{<n}$  and $B\setm A\ne \empt$ then 
$\overline{Z_A}\cap Z_B =\empt$.
\end{claim}

\begin{claim}\label{1.1.1}
For each $i \in n$, $\bigcup_{j \leq i}Y_{j}$ is closed in $X$,
and $\{ Y_{i}: i \in n\}$  forms a partition of $X$.
\end{claim}

\begin{claim}\label{1.1.3}
For each $i \in n$, $Y_{i}$ is left-separated and $\lord{Y_{i}} \leq \oo$.
\end{claim}

These three claims prove our theorem.  Indeed, let $\preceq_i$ be a left-separating 
well-ordering of $Y_i$ in type $\le {\omega}_1$.  Define the well-ordering 
$\preceq$ of $X$ as follows:
\begin{equation*}
y\preceq z \text{ iff \big($y,z\in Y_i$ and $y\preceq_i z$\big)
or $ \big( y\in Y_i$ and $z\in \bigcup_{j>i}Y_j\big )$.  }  
\end{equation*}

Since the $\preceq_i$ are left-separating and  
Claim \ref{1.1.1} holds, 
the initial segments  of $(X,\preceq)$ are closed. 
Moreover, $tp(X,\preceq)\le {\omega}_1\cdot n$ because 
$tp(Y_i,\preceq)\le {\omega}_1.$

\medskip

We first prove claim \ref{1.1.2}.
\begin{proof}[Proof of claim \ref{1.1.2}]
Let $A, B \in[n]^{<n}$ such that $B \backslash A \not= \emptyset.$  Let $k \in B 
\setm A$.  Then $Z_{A} \subset \overline{X_{k}}$ and $Z_{B} \cap 
\overline{X_{k}} = \emptyset$.  Thus $\overline{Z_{A}} \cap Z_{B} = \emptyset$.  
\end{proof}

We now prove claim \ref{1.1.1}.

\begin{proof}[Proof of claim \ref{1.1.1}]
Since 
$X \subset \bigcup \{ Z_{A} | i \in n \wedge A \in [n]^{i} \},$
and $Z_A\cap Z_B=\empt$ for $A\ne B$   by Claim \ref{1.1.2},  the family 
$\{ Y_{i}| i \in n\}$  forms a partition of $X$.

Fix $i \in n$.  Then
\begin{equation*}
\overline{\bigcup_{j\le i}Y_j}\cap (\bigcup_{i<k<n}Y_k)=
\bigcup_{j\le i}\bigcup_{j<k<n}\bigcup_{A\in [n]^i}
\bigcup_{B\in [n]^k}\overline{Z_A}\cap Z_B. 
\end{equation*}
By \ref{1.1.2} this set is empty.
Thus $\bigcup_{j\le i}Y_j$ is closed in $X$.
\end{proof}

The plan for the proof of Claim \ref{1.1.3} is that for each $i \in n$ and each 
$A \in [n]^{i}$ we will well order $Z_{A}$, shuffle $\{ Z_{A}| A \in [n]^{i}\}$ 
together and then show that this well ordering will witness that $Y_{i}$ is left-separated and $\lord{Y_{i}} \leq \oo$.

Let $(N_{\alpha})_{\alpha < \oo}$ be a continuous chain of countable elementary 
submodels such that $X_{i} \in N_{0}$ for each $i \in n$ and $X \subset 
\bigcup_{\alpha < \oo} N_{\alpha}$.

\begin{proof}[Proof of Claim \ref{1.1.3}]
Fix $i \in n$. For each $\alpha < \omega_{1}$ the set $Y_{i} \cap 
N_{\alpha + 1} \setm N_{\alpha}$ is countable and we write it as an $\omega$ 
sequence $\{ z(\alpha, j ) | j \in \omega \}$.  Fix $\alpha < \omega_{1}$. Then $$Y_{i} = 
\{ z(\alpha, j)| j \in \omega \wedge \alpha < \oo  \} $$ and using the lexicographic ordering this well orders 
$Y_{i}$ in order type less than or equal to $\oo$.  We now show that this well 
ordering witnesses that $Y_{i}$ is left-separated.

Since $Z_{A}\cap\overline{Z_{B}}=\empt $ for $A \not= B$ with $A, B \in [n]^{i}$ by Claim
\ref{1.1.2}, it is enough to prove that every $Z_{A}$ is left-separated in 
our well-ordering.

Fix $A\in [n]^i$. 

Let  $z(\alpha , m) \in Z_A$.  Then $z(\alpha, m) \in N_{\alpha +1} \setm 
N_{\alpha}$.  To show that $Z_A$ is left-separated it is sufficient to show 
that $z(\alpha , m ) \notin \overline{Z_A \cap N_{\alpha}}$.  

\begin{claim}
For all $j \notin A$,
$\overline{Z_A \cap N_{\alpha}}  \subset \overline{X_{j} 
\cap N_{\alpha}}$.
\end{claim}

\begin{proof}
Fix $j \notin A.$ Let $x \in Z_A \cap N_{\alpha}$.   So $x \in 
\overline{X_{j}} \cap N_{\alpha}$.  Since $t(X) = \omega$, let $H \in 
[X_{j}]^{\omega} \cap N_{\alpha}$ so that $x \in \overline{H}$.  Since $|H| = 
\omega$ and $H \in N_{\alpha}$ we have that $H \subset N_{\alpha} $.  Thus, $H 
\subset X_{j} \cap N_{\alpha}$ 
and so 
$x\in \overline{X_{j} \cap N_{\alpha}}$.

Thus $Z_A \cap N_{\alpha}  \subset \overline{X_{j} 
\cap N_{\alpha}}$, 
and so $\overline{Z_A \cap N_{\alpha}}  \subset 
\overline{X_{j} \cap N_{\alpha}}$.
\end{proof}

\begin{claim}
$\overline{Z_A \cap N_{\alpha}} \cap Z_A \subset N_{\alpha}$.
\end{claim}

\begin{proof}
Let $x \in \overline{Z_A \cap N_{\alpha}} \cap Z_A$.  
Then $x \notin 
\bigcup_{j \in A} \overline{X_{j}}$.  
Let $j \in (n\setm A)$ such that $x \in X_{j}$.  
Since $ x \in \overline{Z_A \cap N_{\alpha}}$ and $j \notin A$,  by 
the previous claim we have that $x \in \overline{X_{j} \cap N_{\alpha}} \cap 
X_{j}$.  As $X_{j}$ is left-separated, this implies that $x \in N_{\alpha}$. 
\end{proof}

Now let us return to $z(\alpha , m)$.  Since $z(\alpha, m ) \notin N_{\alpha}$ 
and $z(\alpha, m) \in Z_A$, then we have that $z(\alpha, m) \notin 
\overline{Z_A \cap N_{\alpha}}$ which completes the proof.
\end{proof}

\end{proof}

\section{``Left-separated'' property in  c.c.c.  generic extensions}\label{sc:pres}

In this section we prove theorem  \ref{tm:presccc}.

To prove theorem \ref{tm:presccc}(1) we need the following lemma:

\begin{lemma}\label{lm:ccc}
Assume that $X$ is a topological space,  $|X|={\omega}_1$,  
$Q$ is a c.c.c poset
such that 
\begin{align*}
 V^Q\models X=A\cup B, A\subset^{\tiny closed}X, \lord B=\oo. 
\end{align*}
Then there is a partition $X=A^*\cup B^*$ in the ground model  such that 
\begin{enumerate}[(1)]
\item $A^*\subset^{\tiny closed}X$, $\lord {B^*}\le \oo$, 
 \item $1_Q\Vdash A^*\subs {\dot A}$.
\end{enumerate}
 \end{lemma}

\begin{proof}[Proof of lemma \ref{lm:ccc}]
We can assume that the underlying set of $X$ is $\oo$.

Assume that 
\begin{align*}
 1_Q\Vdash \dot f:\oo\to \dot B\text{ is bijection witnessing 
$\lord {\dot B}=\oo$. }
\end{align*}

For ${\alpha}\in \oo$, let
\begin{align*}
 b({\alpha})=\{x\in X :\ \exists q\in Q\ q\Vdash  f({\alpha})=x \}.
\end{align*}
Clearly $b({\alpha})\in \br X;{\omega};$.
For $b\in \br X;{\omega};$
let 
\begin{align*}
 {\alpha}_b=\sup\{{\gamma}<\oo: \exists q\in Q\ q\Vdash \dot f({\gamma})\in b \}.
\end{align*}
Let 
\begin{align*}
 E=\{{\gamma}<\oo: (\forall {\zeta}<{\gamma})\ {\alpha}_{b({\zeta})}<{\gamma}\}.
\end{align*}

Then 
\begin{align}\label{eq:0}
 1_Q\Vdash \dot f''\varepsilon=\varepsilon \cap \dot B  
\end{align}
for all $\varepsilon\in E$.

\begin{claim}
 $E\subs \oo$ is club.
\end{claim}

$E$ is clearly closed. If ${\delta}<\oo$, then define countable  ordinals ${\delta}_n$
and countable subsets $D_n$ of $\oo$ for $n<{\omega}$ as follows:
\begin{itemize}
 \item ${\delta}_0={\delta}$ 
\item $D_n=\bigcup\nolimits_{{\beta}<{\delta}_n}b({\beta})$.
\item ${\delta}_{n+1}={\alpha}_{D_n}$.
\end{itemize}
Then ${\gamma}=\bigcup_{n<{\omega}}{\delta}_n\in E$ and ${\delta}<{\gamma}$.
So $E$ is unbounded as well, which proves the claim. \qed

\smallskip 

By induction on ${\nu}<\omega_2$, we will define families 
\begin{align*}
 \{B^{\nu}_{\xi}:{\xi}<\oo\}
\end{align*}
of pairwise 
disjoint countable subsets of $\oo$ and subsets 
$A^{\nu}\subs \oo$ as follows:
\begin{enumerate}[$\bullet$]
 \item Let $\{\varepsilon_{\xi}:1\le{\xi}<\oo\}$ be the  increasing enumeration 
of $E$, write $\varepsilon_0=0$,
and let
\begin{align*}
B^{0}_{\xi}=\varepsilon_{\xi+1}\setm \varepsilon_{\xi} 
\end{align*}
for ${\xi}<\oo$; 
\item Let 
\begin{align*}
A^{\nu}=X\setm \bigcup\nolimits_{{\xi}<\oo}B^{\nu}_{\xi}. 
\end{align*}
\item If ${\nu}$ is a limit ordinal, let 
\begin{align*}
 B^{\nu}_{\xi}=\bigcap\nolimits_{\eta<{\nu}}B^{\eta}_{\xi}.
\end{align*}
\item if ${\nu}={\mu}+1$, let 
\begin{align*}
B^{\mu+1}_{\xi}=B^{\mu}_{\xi}\setm \overline {
A^\mu\cup\bigcup\nolimits_{{\zeta}<{\xi}}B^{\mu}_\zeta.
}
\end{align*}
\end{enumerate}
\begin{claim}
$1_Q\Vdash {A^{\nu}}\subs \dot A$. 
\end{claim}
\begin{proof}[Proof of the Claim]
By induction on ${\nu}$. 

\noindent {\bf Case 1. }{\em ${\nu}=0$.} 

Since  $A^0=\empt$, the statement is trivial.

\smallskip

\noindent {\bf Case 2. }{\em ${\nu}$ is limit.}

Then $A^{\nu}=\bigcup\nolimits_{{\mu}<{\nu}}A^{\mu}$, so the statement is clear. 

\smallskip

\noindent {\bf Case 3. }{\em ${\nu}={\mu}+1$.}

Assume that $x\in    B^{\mu}_{\xi}\setm B^{\mu+1}_{\xi}$.
Then 
\begin{align}\label{eq:1}
x\in \overline {A^\mu\cup\bigcup\nolimits_{{\zeta}<{\xi}}B^{\mu}_\zeta}. 
\end{align}
Moreover, 
\begin{align}\label{eq:2}
 1_Q\Vdash A^\mu\subs \dot A
\end{align}
by the inductive assumption, and  
\begin{align}\label{eq:3}
\bigcup\nolimits_{{\zeta}<{\xi}}B^{\mu}_{\zeta}\subs 
\bigcup\nolimits_{{\zeta}<{\xi}}B^{0}_{\zeta}=\varepsilon_{\xi}
\end{align}
by the construction.
Putting together \eqref{eq:0}, \eqref{eq:1}, \eqref{eq:2} and \eqref{eq:3} we have 
\begin{align*}
 1_Q\Vdash x\in \overline {\dot A\cup \varepsilon_{\xi}}=
\overline {\dot A}\cup \overline{\dot B\cap \varepsilon_{\xi}}=
\overline {\dot A}\cup \overline{\dot f''\varepsilon_{\xi}}=\\
\dot A\cup
(\dot f''\varepsilon_{\xi})=
\dot A\cup
(\dot B\cap \varepsilon_{\xi}).
\end{align*}
Since $x\notin \varepsilon_{\xi}$, it follows that 
\begin{align*}
 1_Q\Vdash x\in \dot A, 
\end{align*}
which completes the proof of the claim.
\end{proof}

For all ${\xi}<\oo$, the family $\<B^{\nu}_{\xi}:{\nu}<\oot\>$ is decreasing,
so there is ${\nu}_{\xi}<\oot$ such that
\begin{align*}
B^{\nu}_{\xi}=B^{{\nu}_{\xi}}_{\xi}. 
\end{align*}
 for all ${\nu}_{\xi}\le {\nu}<\oot$.
So there is ${\mu}<\oot$
such that 
\begin{align*}
 B^{\mu+1}_{\xi}=B^{\mu}_{\xi} \text{ for all ${\xi}<\oo$}, 
\end{align*}
i.e. 
\begin{align}\label{eq:5}
A^{\mu}\cup\bigcup_{{\zeta}<{\xi}}B^{\mu}_{\zeta}\text{ is closed in $X$ for all ${\xi}<\oo$}
\end{align}
Let 
\begin{align*}
\text{$A^*=A^{\mu}$ and $B^*=X\setm A^*=\bigcup\nolimits_{{\xi}<\oo}B^{\mu}_{\xi}$.} 
\end{align*}
 Then $A^*$ is closed in $X$
by \eqref{eq:5}. 
So $\bigcup_{{\zeta}<{\xi}}B^{\mu}_{\zeta}$ is also closed
in $X\setm A$ for all ${\xi}<\oo$.   
So $B=X\setm A$ is left separated in type $\le \oo$: 
order every $B^{\mu}_{\xi}$
is type $\le \omega$, and declare that, for all ${\zeta}<{\xi}$, 
every element of $B^{\mu}_{\zeta}$ is less
then every element of $B^{\mu}_{\xi}$:
\begin{align*}
 B^{\mu}_0<B^1_{\mu}<\dots < B^{\mu}_{\zeta}<\dots  < B^{\mu}_{\xi}<\dots
\end{align*}
 
\end{proof}

\begin{proof}[Proof of theorem \ref{tm:presccc}(1)]
By induction on $n\in \omega $ we prove the following statement:
\begin{enumerate}[($*_n$)]
 \item 
 If $X$ is a topological space, $Q$ is a c.c.c poset such that 
\begin{align*}
 V^Q\models \lord{X}\le\omega_1\cdot n 
\end{align*}
then  $X$ is left-separated in even in $V$, moreover
\begin{align*}
 \lord{X}^V\le \omega_1\cdot n.      
\end{align*}

\end{enumerate}
Assume that ($*_m$) holds for $m<n$, and 
$X$ is a topological space, $Q$ is a c.c.c poset such that 
\begin{align*}
 V^Q\models \lord{X}\le\omega_1\cdot n 
\end{align*}
We can assume that in $V^Q$ the space $X$ has a left-separating well-ordering
in type $\omega_1\cdot n$. (If $\lord{X}<\omega_1\cdot n$, add 
$\omega_1$ isolated points to $X$).
Let $B$ be the last $\oo$ many points of $X$ in that well-ordering and 
$A=X\setm A$.  

By lemma \ref{lm:ccc}
then there is a partition $X=C\cup D$ in the ground model  such that 
\begin{enumerate}[(1)]
\item $C\subset^{\tiny closed}X$, $\lord D\le \oo$, 
 \item $1_Q\Vdash C\subs A$.
\end{enumerate}
Then 
\begin{equation*}
\text{$V^Q\models \lord{C}\le \lord{A}\le \omega_1\cdot (n-1)$,}
\end{equation*}
so by the inductive assumption, we have $\lord{C}\le \omega_1\cdot (n-1)$ 
in the ground model.
Thus $\lord{X}\le \lord{C}+\lord{D}=\omega_1\cdot (n-1)+\omega_1=
\omega_1\cdot n$ because $C$ is closed in $X$.
So we proved theorem \ref{tm:presccc}(1).
\end{proof}

Next we show that, consistently, a non-left-separated space may be 
left-separated in some c.c.c generic extension of the ground model. 
The construction of the model of theorem \ref{tm:smtsd}
is quite complicated, so first we consider an other example which is 
much simpler, but also much weaker. {\em first we give a Hausdorff example 
 which is 
much simpler/ which can be obtained much simpler.}

Let us recall that an uncountable subset of the reals is called {\em Luzin}
if it has countable intersection with every meager set. 
Mahlo and Luzin, independently,  proved that if CH holds, then there is a 
Luzin set.

\begin{theorem}
If there is a Luzin set, then 
there is a Hausdorff topological space $X$ such that 
$\lord X=\infty$, but $\lord X=\omega_1\cdot  \omega $ is some 
c.c.c generic  extension.    
\end{theorem}

\begin{proof}
Since there is a Luzin set $Z$, there is a Luzin $Y$ 
such that $|Z\cap V|=\oo$ for all non-empty euclidean open set $V$.   
 We can assume that $Y$
can not contain any rational numbers. 

The underlying set of our space is $X=\mbb Q\cup Y$, and the 
topology  $\tau$ on $X$ is the refinement of the euclidean topology 
by declaring the countable subsets of $Y$ to be closed. 
So the following family $\mc B$ is a base of $X$:
\begin{align*}
 \mc B=\{U\setm Z: U\text{ is a euclidean open set }, Z\in \br Y;{\omega};  \}.
\end{align*}

 Assume that $F\subs Y$ is  nowhere dense
in the topology $\tau$. 
Then there is  euclidean dense open set $U$ and a countable subset $Z$
of $Y$ such that $F\cap (U\setm Z)=\empt$.
So $F\cap U\subs Z$. Since $|F\setm U|\le {\omega}$ because $Y$ is Luzin,
we have $|F|\le {\omega}$. 

Since the countable subsets of $Y$ are closed and $\Delta(Y)=\oo$,
it follows that $F\subs Y$ is nowhere dense in $\tau$ iff $Y$ is countable.

So we can apply lemma \ref{lm:baire_not_left_sep} to show that  
 $X$ is not left-separated: 
let  $A=\mbb Q$ and $B=Y$. Since the union of countably many 
nowhere dense subset of $Y$ is nowhere dense, we have that $X$ 
is not left separated.

Consider a c.c.c generic extension $W$ of the ground model in which 
Martin's Axiom holds. Then in $W$ we have countable many dense open sets 
$\{U_n:n\in {\omega}\}$ such that 
\begin{align*}
\mbb Q\subs   \bigcap\nolimits_{n\in {\omega}}U_n\subs \mbb R\setm Y. 
\end{align*}
Let 
$$
F_n=(Y\setm \bigcap\nolimits_{i\in n}U_i)\setm \bigcup\nolimits_{m<n} F_m.
$$
Then $F_n$ is a closed subset of $X$ and it is left-separated in type $\oo$
because the countable subsets of $Y$ are all closed.
Write $\mbb Q=\{q_n:n\in {\omega}\}$.
Then we have a left-separating order of $X$ in type $\oo\cdot {\omega}$:
\begin{align*}
 F_0< q_0< F_1 < q_1 < F_1 < \dots
\end{align*}
where $F_n$ is ordered in type $\oo$.
\end{proof}

Instead of theorem \ref{tm:presccc}(2) we will prove theorem \ref{tm:smtsd},
which is a stronger statement.    
To formulate that theorem. we need some preparation.

First we recall a definition from the proof of \cite[Theorem 3.5]{JSSz}.
\begin{definition}
For each uncountable cardinal ${\kappa}$ define the poset 
$\pcalk=\<\Pko,\le\>$  as follows.

A quadruple $\anfg$ is in $\Pkoo$ provided
$(1)$--$(5)$ below hold:
\begin{enumerate}[(1)]
\item $A\in\br {\kappa};<{\omega};$,
\item $n\in {\omega}$,
\item $f$ and  $g$ are functions,
\item $f:A\times A\times n\to 2$,
\item $g:A\times n\times A\times n\to 3$,
\end{enumerate}

For $p\in \Pkoo$ we write $p=\anfgi p;$.
If $p, q\in \Pkoo$ we set
\begin{align*}
\text{$p\le q$ iff   $f^p\supseteq f^q$ and $g^p\supseteq g^q$. 
} 
\end{align*}
If $p\in \Pkoo$, ${\alpha}\in \ap$, $i<\np$ set
\begin{align*}
\text{$
U({\alpha},i)=\up({\alpha},i)=\{{\beta}\in \ap:\fp({\beta},{\alpha},i)=1\}
$.
} 
\end{align*}
A quadruple $\anfg\in \Pkoo$ is in $\Pko$ iff (i)--(iv) below are also 
satisfied:
\begin{enumerate}[(i)]
\item $\forall {\alpha}\in A$ $\forall i<n$ ${\alpha}\in U({\alpha},i)$,
\item $\forall {\alpha}\in A$ $\forall i<j<n$  
$U({\alpha},j)\subs U({\alpha},i)$,
\item $\forall \{{\alpha},{\beta}\}\in \br A;2;$ $\forall i,j<n$
\newline
\begin{tabular}{rcl}
$U({\alpha},i)\subs U({\beta},j)$&iff& $g({\alpha},i,{\beta},j)=0$,\\
$U({\alpha},i)\cap U({\beta},j)=\empt$&iff& $g({\alpha},i,{\beta},j)=1$.
\end{tabular}
\item $\forall \{{\alpha},{\beta}\}\in \br A;2;$ $\forall i,j<n$\newline
if ${\alpha}\in U({\beta},j)$ and ${\beta}\in U({\alpha},i)$
then $g({\alpha},i,{\beta},j)=2$.
\end{enumerate}
Finally let  $\pcalk=\<\Pko,\le\>$.

\end{definition}

Using this definition we can formulate our next result:

\begin{theorem}\label{tm:smtsd}
For each uncountable cardinal ${\kappa}$, 
the poset $V^{\pcalk}$ satisfies c.c.c and 
 in $V^{\pcalk}$ 
there is a   0-dimensional,
first countable topological space $X=\<{\kappa},{\tau}\>$ and
there is a c.c.c posets
$\qsd$   satisfying the following conditions: 
\begin{enumerate}[(a)]
\item $V^{{\pcalk}}\models$ ``{\em $R(X)=\omega$},
so $X$ does not contain uncountable left-separated subspaces.
\item $V^{\pcalk*\qsd}\models$ ``{\em $X$ is left separated in type $
\kappa\cdot \omega$}'', 
\end{enumerate}
\end{theorem}

\begin{proof}[Proof of theorem \ref{tm:smtsd}]
It was proved in the proof of \cite[Theorem 3.5]{JSSz}
that the poset $V^{\pcalk}$ satisfies c.c.c.

Let ${\mc G}$ be the $\pcalk$ generic filter and
let $F=\bigcup\{f^p:p\in{\mc G}\}$.
For each ${\alpha}<{\kappa}$ and $n\in{\omega}$ let
$V({\alpha},i)=\{{\beta}<{\kappa}:F({\beta},{\alpha},i)=1\}$.
Put ${\mc B}_{\alpha}=\{V({\alpha},i):i<{\kappa}\}$ and
${\mc B}=\bigcup\{{\mc B}_{\alpha}:{\alpha}<{\kappa}\}$.
By standard density arguments we can see that
${\mc B}$ is  base of a first countable, 0-dimensional $T_2$  space 
$X=\<{\kappa},{\tau}\>$.

\cite[Theorem 3.5]{JSSz} claimed that $R(X)={\omega}$,
so $X$ can not contain uncountable left-separated subspaces.

To prove that $X$ is left separated in type 
${\kappa}\cdot {\omega}$ in some c.c.c generic extension $V^{\pcalk}*Q$
 define the poset $\qsd$ 
in $V^{\pcalk}$ as follows:

A triple $\bde$ is in  $\qsd$ iff
\begin{enumerate}[(a)]
\item $B\in\br{\kappa};<{\omega};$,
\item $d:B\to {\omega}$,
\item $e:B\to {\omega}$,
\item for each $ \{{\alpha},{\beta}\}\in \br B;2;$ if $d({\alpha})\le d({\beta})$
then ${\alpha}\notin V({\beta},e({\beta}))$.
\end{enumerate}

The orderings on $\qsd$  is defined in the straightforward 
way,
\begin{align*}
\bdei 0;\le \bdei 1;\text{ iff } d^0\supset d^1 \text{ and } 
e^0\supset e^1.
\end{align*}
If $q$ and $r$ are compatible elements of $Q$, then denote by
$q\land r$ their greatest lower bound in $Q$.
\begin{lemma}\label{lm:Qccc}
$\pcalk*\qsd$ satisfies c.c.c.
\end{lemma}
To prove this lemma
we need to recall two more 
definitions and a lemma from \cite{JSSz}. 
\begin{definition}[{\cite[Definition 3.6]{JSSz}}]
Assume that $p_i=\anfgi i;\in \Pkoo$ for $i\in 2$. We say that $p_0$ and $p_1$
are {\em twins} iff  $n^0=n^1$, $|A^0|=|A^1|$ and taking $n=n^0$ and
denoting by ${\sigma}$ the unique $<$-preserving bijection between $A^0$ and
$A^1$ we have
\begin{enumerate}[(i)]
\item ${\sigma}\restriction {A^0\cap A^1}=\id_{A^0\cap A^1}$.
\item ${\sigma}$ is an isomorphism between $p_0$ and $p_1$, i.e.
$\forall {\alpha},{\beta}\in A^0$, $\forall i,j<n$
\begin{enumerate}[({ii}-a)]
\item $f^0({\alpha},{\beta},i)=f^1({\sigma}({\alpha}),{\sigma}({\beta}),i)$,
\item $g^0({\alpha},i,{\beta},j)=
g^1({\sigma}({\alpha}),i,{\sigma}({\beta}),j)$,
\end{enumerate}
\end{enumerate}
We say that ${\sigma}$ is the {\em twin function} of $p_0$ and $p_1$.
Define the {\em smashing function} $\sbar$ of $p_0$ and $p_1$ as follows:
 $\sbar={\sigma}\cup \id_{A_1}$.  
The function $\sstr$ defined by the formula 
$\sstr={\sigma}\cup {\sigma}^{-1}\restriction {A_1}$ is called the
{\em exchange function} of $p_0$ and $p_1$.
\end{definition}

\begin{definition}
\label{def:eps_amalg}
Assume that $p_0$ and $p_1$ are twins and 
${\varepsilon}:A^{p_1}\setm A^{p_0}\to 2$. A common extension 
$q\in\Pko$ of $p_0$ and $p_1$  is called an
{\em ${\varepsilon}$-amalgamation} of the twins provided
\begin{equation}
\forall {\alpha}\in A^{p_0}\triangle A^{p_1}\ 
f^q({\alpha},\sstr({\alpha}),i)={\varepsilon}(\sbar({\alpha})).
\end{equation}
The notion of an {${\varepsilon}$-amalgamation} was introduced in 
{\cite[Definition 3.7]{JSSz}}.

An {${\varepsilon}$-amalgamation} is a {\em strong ${\varepsilon}$-amalgamation}
if  
\begin{multline}\tag{$*$}\label{strong}
\forall \{{\alpha},\beta\}\in \br A^{p_0}\cup  A^{p_1};2;\ \forall i<n^{p_0}  \\
\text{ if } \sstr({\alpha})\ne \sstr({\beta})\text{ then }
f^q({\alpha},{\beta},i)=f^{\mu}(\sstr({\alpha}),\sstr({\beta}),i).
\end{multline}
\end{definition}

\begin{lemma}
\label{lm:twins}
If $p_0$,  $p_1\in\pcalk$ are twins and 
${\varepsilon}:A^{p_1}\setm A^{p_0}\to 2$, then $p_0$
and $p_1$ have a strong  ${\varepsilon}$-amalgamation in $\Pko$
such that .
\end{lemma}

\begin{proof}
In \cite[Lemma 3.8]{JSSz} we proved that $p_0$ and $p_1$
have an ${\varepsilon}$-amalgamation $q$  in $\Pko$.
However, the condition $q$, which was defined in the first paragraph 
of the proof of \cite[Lemma 3.8]{JSSz} is actually
a strong ${\varepsilon}$-amalgamation.
\end{proof}

Now we are ready to prove our lemma.

\begin{proof}[Proof of Lemma \ref{lm:Qccc}]
Let $\<\<p_{\nu},q_{\nu}\>:{\nu}<\oo\>\subs \pcalk*\qsd$.
We can assume that $p_{\nu}$ decides $q_{\nu}$. Write
$p_{\nu}=\anfgi {\nu};$ and $q_{\nu}=\bdei {\nu};$.
By standard density arguments we can assume that
$A^{\nu}\supset B^{\nu}$. Applying standard $\Delta$-system and counting
arguments  we can find 
$\{{\nu},{\mu}\}\in \br \oo;2;$ such that $p_{\nu}$ and $p_{\mu}$ are twins 
and denoting by ${\sigma}$ the twin function of $p_{\nu}$ and
$p_{\mu}$ we have 
\begin{enumerate}[(i)]\addtocounter{enumi}{2}
 \item $B^{\mu}=\sigma''B^{\nu}$,
\item 

 $d^{\nu}({\alpha})=d^{\mu}({\sigma}({\alpha}))$ for each ${\alpha}\in B^{\nu}$,
and 
\item $e^{\nu}({\alpha})=e^{\mu}({\sigma}({\alpha}))$
for each ${\alpha}\in B^{\nu}$.
\end{enumerate}

Define the function ${\varepsilon}^0:A^{\mu}\setm A^{\nu}\to 2$ by the
equation ${\varepsilon}^0({\alpha})=0$. By Lemma \ref{lm:twins}
the conditions $p^{\nu}$ and $p^{\mu}$ have an 
${\varepsilon}^0$-amalgamation $p$.
We claim that 
$$
p\Vdash\text{$q_{\nu}$ and  $q_{\mu}$ are compatible in  $\qsd$,}
$$
i.e
$$
p\Vdash \<B^{\nu}\cup B^{\mu}, d^{\nu}\cup d^{\mu}, e^{\mu}\cup e^{\mu}\>
\text{has properties (a)--(d)}.
$$
(a) clearly holds. Since $\sigma({\alpha})={\alpha}$ for each 
${\alpha}\in B^{\nu}\cap B^{\mu}\subs A^{\nu}\cap A^{\mu}$,
assumption (iv) implies that $d^{\nu}\cup d^{\mu}$ is a function, and 
assumption (v) implies that $e^{\nu}\cup e^{\mu}$ is a function, 
and so (b) and (d) also hold.

To check (d) it is enough  to show that 
if ${\alpha}\in B^{\nu}$,
${\beta}\in B^{\mu}$, $d^{\nu}({\alpha})\le d^{\mu}({\beta})$, then
$p\Vdash   {\alpha}\notin V({\beta},e^{\mu}({\beta}))$, i.e. 
\begin{displaymath}
\tag{$\star$}\label{star}
f^p({\alpha},\beta, e^{\mu}({\beta}))=0. 
\end{displaymath}
Assume first that  $\sigma({\alpha})=\sstr({\alpha})\ne {\beta}$.  
Then $$d=d^{\mu}(\sigma({\alpha}))=d^{\nu}({\alpha})$$ by (iv),
$$e=e^{\mu}(\sigma({\alpha}))=e^{\nu}({\alpha})$$  by  (v),
and so 
$d^{\mu}(\sigma({\alpha}))\le d^{\mu}({\beta})$. Thus,  by (d),
\begin{displaymath}
    p_{\mu}\Vdash \sigma({\alpha})\notin V({\beta},e^{\mu}({\beta})),
\end{displaymath} 
i.e.
\begin{displaymath}
 f^{\mu}(\sigma({\alpha}), {\beta},e^{\mu}({\beta}))=0.
\end{displaymath}
Since $p$ is a strong $\varepsilon^0$-amalgamation, condition \eqref{strong}
implies that 
\begin{displaymath}
 f^p({\alpha}, {\beta},e^{\mu}({\beta}))=f^{\mu}(\sigma({\alpha}), {\beta},e^{\mu}({\beta}))=0,
\end{displaymath}
i.e. \eqref{star} holds.

If $\sstr({\alpha})={\beta}$, then $p\Vdash   {\alpha}\notin V({\beta},e^{\mu}({\beta}))$
because $p$ is an ${\varepsilon}^0$-amalgamation, and so 
\begin{displaymath}
f^p({\alpha},\sstr({\alpha}), e^{\mu}({\beta}))=\varepsilon(\sbar({\alpha}))=0.
\end{displaymath}
\end{proof}

\begin{lemma}
$$
V^{\pcalk*\qsd}\models\mbox{``$X$ is left-separated in type $\kappa\cdot \omega$''}
$$
\end{lemma}

\begin{proof}
Let ${\mc H}$ be the $\qsd$-generic filter over $V^{\pcalk}$.
Set $d=\bigcup\{d^q:q\in {\mc H}\}$ and
$e=\bigcup\{e^q:q\in {\mc H}\}$. By standard density arguments 
the domains of the functions  $d$ and $e$ are ${\kappa}$.
We have 
\begin{align}
 V(x,e(x))\cap \big(\bigcup\nolimits_{i\le d(x)}d^{-1}\{i\}\big )=\{x\},
\end{align}
 by (d), so
 $d^{-1}(n)$ is discrete 
and $\overline{d^{-1}\{n\}}\cap d^{-1}n=\empt$
for each $n\in {\omega}$.
So we can construct  a left-separating ordering of $X$ in type 
${\kappa}\cdot {\omega}$ as follows:  
consider a well-ordering $\prec$ of $X$ such that 
  $d^{-1}\{n\}$  is well-ordered in type ${\kappa}$ for $n\in {\omega}$, and 
\begin{align*}
d^{-1}\{0\}\prec d^{-1}\{1\}\prec \dots <d^{-1}\{n\}\prec d^{-1}\{n+1\}\prec \dots. 
\end{align*}
\end{proof}
This completes the proof of theorem \ref{tm:smtsd}.
\end{proof}

\end{document}